\documentclass[11pt]{amsart}

\usepackage[usenames]{color}
\usepackage{hyperref}
\usepackage{amsmath,amsfonts,amssymb}
\numberwithin{equation}{section}

\title{Characterising derivations from the disc algebra to its dual}
\author{Y. Choi}
\address{D\'epartement de math\'ematiques
et de statistique\\
Pavillon Alexandre-Vachon\\
Universit\'e Laval\\
Qu\'ebec, QC \\
Canada, G1V 0A6%
}
\email{y.choi.97@cantab.net}
\author{M. J. Heath}
\address{
Departamento de Matem\'atica \\
Instituto Superior T\'ecnico  \\
Av. Rovisco Pais \\
1049-001 Lisboa \\
Portugal%
}
\email{mheath@math.ist.utl.pt}
\thanks{Second author supported by post-doctoral grant SFRH/BPD/40762/2007 from FCT (Portugal).}
\keywords{
Derivation, disc algebra, Hardy space
}

\subjclass[2010]{Primary 46J15; Secondary 30H10, 47B47}

\usepackage{amsthm}
\newcounter{pulse}
\numberwithin{pulse}{section}
\theoremstyle{plain}
\newtheorem{prop}[pulse]{Proposition}
\newtheorem{lem}[pulse]{Lemma}
\newtheorem{cor}[pulse]{Corollary}
\newtheorem{thm}[pulse]{Theorem}
\theoremstyle{definition}

\theoremstyle{remark}
\newtheorem{rem}[pulse]{Remark}

\newenvironment{newnum}{%
\begin{enumerate}

}{\end{enumerate}\ignorespacesafterend}

\usepackage{verbatim}

\newcommand{\dt}[1]{\emph{#1}\/}  

\renewcommand{\dt}[1]{\textcolor{Bittersweet}{\sf#1}}

\newcommand{\Abs}[1]{\left\vert#1\right\vert}
\newcommand{\abs}[1]{\vert#1\vert}
\newcommand{\norm}[1]{\Vert#1\Vert}
\newcommand{\snorm}[1]{\Vert#1\Vert_\infty}

\renewcommand{\Re}{\mathop{\sf Re}\nolimits}
\renewcommand{\Im}{\mathop{\sf Im}\nolimits}

\newcommand{\Cplx}{{\mathbb C}}

\newcommand{\D}{\mathbb D}
\newcommand{\Disc}{\D}
\newcommand{\T}{\mathbb T}

\newcommand{\AD}{A({\mathbb D})}  

\newcommand{\sB}{{\sf B}}
\newcommand{\sE}{{\sf E}}
\newcommand{\Der}{\mathcal D}
\newcommand{\defeq}{:=}
\newcommand{\id}[1][]{\mathbf{1}_{#1}}

\newcommand{\logweight}[1]{{\log\frac{1}{\abs{#1}}\operatorname{dA}({#1})}}
\newcommand{\Conj}{\operatorname{Conj}}
\newcommand{\wLOG}{without loss of generality}
\newcommand{\st}{\,:\,} 

\renewcommand{\d}{\operatorname{d}}
\renewcommand{\d}[1]{\operatorname{d\ensuremath{#1}}}
\newcommand{\darc}[1]{\operatorname{\vert d#1\vert}} 

\newcommand{\pair}[2]{{\left\langle#1,#2\right\rangle}}

\begin{document}

\begin{abstract}
We show that the space of all bounded derivations from the disc algebra into its dual can be identified with the Hardy space~$H^1$; using this, we infer that all such derivations are compact. Also, given a fixed derivation~$D$, we construct a finite, positive Borel measure $\mu_D$ on the closed disc, such that $D$ factors through $L^2(\mu_D)$. Such a measure is known to exist, for any bounded linear map from the disc algebra to its dual, by results of Bourgain and Pietsch, but these results are highly non-constructive.
\end{abstract}

\maketitle
\begin{section}{Introduction}
Although derivations from Banach algebras to particular coefficient modules have been much studied, interest has usually focused on the \emph{existence} or otherwise of non-trivial bounded derivations, rather than their characterisation. Even in the special case of derivations from a commutative Banach algebra to its dual (as in~\cite{BCD_Lip} or \cite{Run_Kahl} for example), there are comparatively few examples where the space of such derivations is explicitly characterised or parametrised.

For uniform algebras, very little is known: the only examples with a complete characterisation are the trivial uniform algebras $C(X)$ -- and these have no nonzero bounded derivations into \emph{any} dual $C(X)$-bimodule.
In this paper we provide the first example of 
a non-trivial uniform algebra -- namely, the disc algebra -- where
the space of derivations to its dual can be completely characterised.
For the disc algebra, this is equivalent to characterising those complex Borel measures $\mu$ on the unit circle, for which the bilinear functional $(f,g)\mapsto\int_\T f'g\d{\mu}$ on the space of analytic polynomials is bounded in the supremum norm on the unit disc.

Indeed, our proof identifies this space, in a natural way, with a rather well-known space of functions on the unit circle. 
Some of our techniques should be applicable to other uniform algebras, but for sake of brevity we do not pursue this here.

As a by-product of our characterisation, we obtain a short proof that every bounded derivation from the disc algebra to its dual is automatically compact, resolving a question raised in~\cite{SEM_PhD}. 
Finally, using some ideas from the proof of our characterisation, we show how one can construct a `Pietsch control measure' for a given derivation, which witnesses the (known) fact that each such derivation is $2$-summing. Our construction is explicit and does not rely on the previous results of Bourgain on arbitrary linear maps from the disc algebra to its dual.

\subsection*{Notation and other preliminaries}
Throughout, `derivation' means `bounded derivation'.
Given a Banach algebra $A$, we denote by $\Der(A)$ the space of bounded derivations from $A$ to is dual space~$A^*$\/.

We briefly review the (mostly standard) notation used in describing the disc algebra and related spaces.
\begin{itemize}
\item $\Delta$ denotes the closed unit disc in the complex plane $\Cplx$\/; $\Disc$ denotes its interior, and $\T$ its boundary. We write $\AD$ for the \dt{disc algebra}, that is, the space of all functions which are continuous on $\Delta$ and analytic on $\Disc$\/. Via the maximum modulus principle this can be regarded as a closed subalgebra of
 $C(\T)$, and we shall do this without further comment.
\item $\darc{z}$ denotes arc length measure on~$\T$\/, normalised so that the length of $\T$ is $2\pi$. For $1\leq p < \infty$\/, the norm on $L^p(\T)$ is denoted by $\norm{\cdot}_p$\/, and normalised so that $\norm{\id}_p = (2\pi)^{1/p}$\/. 
\item For $1\leq p \leq \infty$\/, we denote by $H^p$ the familiar Hardy space on $\Disc$\/. We will identify (via radial limits) elements of $H^p$ with those functions in $L^p(\T)$ whose negative Fourier coefficients are zero. It is also convenient to write $H^p_0$ for the subspace
\[ H^p_0 \defeq \left\{ f\in H^p\st f(0)=\int_\T f(z) \darc{z} = 0 \right\}\,. \]  
\end{itemize}
Finally, if $X$ is a set of complex-valued functions, we denote by $\Conj X$ the set $\left\{\overline{f}\st f\in X\right\}$\/.

\end{section}

\begin{section}{The characterisation}
Our characterisation is as follows:
\begin{thm}\label{t:characterise}
 Let $D\in\Der (\AD)$. Then there is a unique $h_D\in \Conj H^1_0$ such that 
\[ D(f)(\id)=\int_\T f(z)h_D(z) \darc{z} \qquad\text{ for all $f\in\AD$.} \]
Furthermore, $\sE:D\mapsto h_D$ defines a bounded, linear isomorphism from $\Der(\AD)$ onto $\Conj H^1_0$.
\end{thm}
As we shall see, the tools required to prove this result already exist in the literature. It is their combination which appears to be new.

\subsection{Everything except surjectivity of $\sE$}
In this subsection we prove most of Theorem \ref{t:characterise}; more precisely we prove the following.

\begin{prop}\label{t:butsurj}
 Let $D\in\Der (\AD)$. Then there is a unique $h_D\in \Conj H^1_0$ such that 
\[D_h(f)(\id)=\int_\T f(z)h(z) \darc{z} \qquad\text{ for all $f\in\AD$.}\]
Furthermore $\sE:D\mapsto h_D$ defines an injective, bounded, linear map from $\Der(\AD)$ to $\Conj H^1_0$.
\end{prop}
To prove this we shall require some lemmas, and some auxiliary notation.
Define $\sB: \Der(\AD)\to \AD^*$ by
\[ \sB(D)(f) \defeq D(f)(\id)\qquad(f\in\AD). \]

\begin{lem}\label{l:sB}
$\sB$ is a continuous, injective linear map.
\end{lem}
\begin{proof}
It is clear that $\sB$ is linear and contractive. Suppose $D\in\Der(\AD)$ and that $\sB(D)=0$\/. Given analytic polynomials $f$ and $g$\/, let $u$ be the polynomial with constant term zero and whose derivative equals $f'g$\/. Then
\[ D(f)(g) = D(Z)(f'g) = D(u)(\id) = \sB(D)(u) = 0 \;;\]
and so by continuity $D=0$\/. Thus $\sB$ is injective.
\end{proof}

It follows from work of Morris in \cite{SEM_PhD} that $\sB$ factors through the natural inclusion of $\Conj H^1_0$ into $\AD^*$\/.
The main ingredient is the following proposition.

\begin{prop}\label{p:fromunifalg}
Let $A$ be a unital uniform algebra on a compact space $X$. Let $M$ be a symmetric, contractive Banach $A$-bimodule, and let $D:A\to M$ be a bounded derivation.
\begin{newnum}
\item If $f\in A\cap \Conj{A}$ then $D(f)=0$\/.
\item If $h\in A$ then $\norm{D(h)}\leq 2e\norm{D}\snorm{\Re h}$\/.
\item If $f,g\in A$ then $\norm{D(f)}\leq 2e \norm{D} \snorm{f-\overline{g}}$\/.
\end{newnum} 
\end{prop}

As \cite{SEM_PhD} is not easily available, and 
since the proof of Proposition~\ref{p:fromunifalg} is short and instructive, we give the argument below, paraphrased slightly from the original wording.

\begin{proof}
Let $a\in A$\/. Since
$M$ is symmetric,
an easy induction
using the Leibniz identity
 shows that $D(a^n)=na^{n-1}\cdot D(a)$ for all $n\geq 1$;
on taking $a=\id[A]$ and $n=2,3$ we have
\[ D(\id[A]) = 2\cdot\id[A]\cdot D(\id[A]) = 3\cdot\id[A]^2\cdot D(\id[A]) = 3\cdot\id[A]\cdot D(\id[A]), \]
which forces $D(\id[A])=0$. 
It follows, by considering the power series expansion for $\exp$\/, that $D(e^a)= e^a\cdot D(a)$\/, and therefore
\begin{equation}\label{eq:exp-trick}
e^{-a}D(e^a)= D(a) \qquad\text{ for all $a\in A$\/,}
\end{equation}
(Note that this does not use the fact that $A$ is a uniform algebra.)

Now let $f\in A\cap \Conj{A}$\/. For every $t>0$ we have $\snorm{e^{itf}}\leq 1$ and $\snorm{e^{-itf}}\leq 1$\/; so by \eqref{eq:exp-trick} we have $\norm{D(tf)}\leq \norm{D}$\/. Since $t$ is arbitrary, $D(f)=0$\/, and (i)~is proved.

Now suppose $h\in A$\/. If $\Re h = 0$ then
$ih\in A\cap\Conj{A}$, and hence by part~(i) $iD(h)=D(ih)=0$.
We may thus assume \wLOG\ that $\snorm{\Re h} > 0$\/. By \eqref{eq:exp-trick}, for every $a\in A$ we have
\[ \norm{D(a)} \leq \norm{D} \snorm{e^a}\norm{D} \snorm{e^{-a}}
= \norm{D} \snorm{e^{\Re a}} \snorm{e^{-\Re a}}
\leq \norm{D} \exp (2\snorm{\Re a})\,. \]
Putting $a= (2\snorm{\Re h})^{-1}h$ we get $(2\snorm{\Re h})^{-1} D(h) \leq e$, and (ii)~follows.

Finally, given $f,g\in A$, note that
\[ \Re(f-\overline{g})=\Re(f-g) \quad,\quad \Im(f-\overline{g}) = \Re( -i(f+g))\,.\]
Applying (ii) with $h=f-g$ and then with $h=-i(f+g)$ gives
\[ \begin{aligned}
\norm{D(f)-D(g)} & \leq 2e\norm{D} \snorm{\Re(f-\overline{g})y} \leq 2e\norm{D} \snorm{f-\overline{g}}\,, \\
\norm{D(f)+D(g)} &
\leq 2e\norm{D} \snorm{\Im(f-\overline{g})} \leq 2e\norm{D} \snorm{f-\overline{g}}\,.
\end{aligned} \]
Combining these last two inequalities and using the triangle inequality, we obtain~(iii).
\end{proof}
\begin{proof}[Proof of Proposition \ref{t:butsurj}]
Take $A=\AD$ and $M=\AD^*$ in Proposition \ref{p:fromunifalg}. Then for each $D\in\Der(\AD)$, the functional $\sB(D)\in\AD^*$ extends to a continuous linear functional on $C(\T)/\Conj \AD$\/, with norm at most $2e\norm{D}$\/. 
By the F.~\&~M. Riesz Theorem,  $(C(\T)/\Conj \AD)^*$ consists precisely of functionals of the form 
\[\psi:f+\Conj \AD\mapsto \int_\T fh_\psi \darc{z}, \]
where $h_\psi\in \Conj H^1_0$ (necessarily unique) and $\norm\psi=\norm{h_\psi}_{L^1}$.
Set $\sE(D):=h_{\sB(D)}$. We then have  $\sE(D) \in \Conj H^1_0$\/, and $\norm{\sE(D)}_{L^1}\leq 2e \norm{D}$\/. That $\sE$ is linear and injective follows from Lemma~\ref{l:sB}.
\end{proof}

\subsection{Surjectivity}
To go further, we shall use some standard (albeit non-trivial) results from the theory of BMOA. Recall that a function $f\in H^2$ lies in the space BMOA if and only if it satisfies one of the following three equivalent conditions:
\begin{newnum}
\item\label{item:BMO}
 there exists a constant $C_1$ such that
\[ \frac{1}{\abs{I}} \int_I \abs{f- f_I} \darc{z} \leq C_1 \quad\text{ for every arc $I\subset \T$\/,} \]
where $f_I$ denotes the average of $f$ over $I$\/.
\item\label{item:dual}
 there exists a constant $C_2$ such that $\abs{ \int_\T f \overline{h} \darc{z} } \leq C_2\norm{h}_1$ for all $h\in H^1_0$\/.
\item\label{item:carleson}
 there exists a constant $C_3$ such that
\[ \int_\Disc \abs{f'(z)k(z)}^2 (1-\abs{z})^2 \d{A}(z) \leq C_3^2 \int_\T \abs{k}^2 \darc{z} \qquad\text{ for all $k\in H^2$\/,} \]
where $\d{A}$ is the usual area measure on $\Disc$.
\end{newnum}
By taking the least possible constants $C_1$\/, $C_2$ and $C_3$ which satisfy the above, one obtains three mutually equivalent
seminorms
 $\norm{\cdot}_{(1)}$\/, $\norm{\cdot}_{(2)}$ and $\norm{\cdot}_{(3)}$ on the space BMOA.
(The kernel of each of these seminorms is the space of constant functions.)

\begin{rem}
Strictly speaking, \ref{item:BMO} is the true definition. The equivalence of \ref{item:BMO} and \ref{item:dual} is Fefferman's celebrated duality theorem, while the equivalence of \ref{item:BMO} and \ref{item:carleson} is essentially the Carleson embedding theorem in a special case.
Both results are well-known to specialists, and proofs can be found in \cite{Gar_BAF}, although assembling them from scattered parts takes some work on the reader's part.
(The duality theorem is covered in \cite[Ch.~VI, \S4]{Gar_BAF} -- see the discussion at the top of p.~240 
for how one recovers the \emph{complex} form of Fefferman's theorem from the \emph{real} form -- while a proof of the Carleson multiplier theorem can be extracted from a combination of \cite[Theorem I.5.6]{Gar_BAF} and \cite[Theorem VI.3.4]{Gar_BAF}.)
\end{rem}

We can now give the proof of Theorem \ref{t:characterise}. The remainder of the argument is essentially contained, phrased differently, in results from~\cite{AlPe96} (see in particular Theorem~3.4 of that article). However, the required facts about BMOA predate \cite{AlPe96}, and it is easier to use them directly, rather than try to deduce our theorem from that paper's results.


\begin{proof}[Proof of Theorem~\ref{t:characterise}]
Having proven Proposition \ref{t:butsurj}, it only remains to show that $\sE$ is surjective. Let $h\in H^1_0$\/. We claim that there is a constant $K>0$ such that, for any $f$ and $g$ which are holomorphic on a neighbourhood of $\Delta$\/, we have
\begin{equation}\label{eq:keybound}
\int_\T u(z)\overline{h(z)} \darc{z} \leq K \snorm{f}\snorm{g}
\end{equation}
where $u$ is defined by $u(0)=0$ and $u'=f'g$\/. If this is the case, then we can extend the map $(f,g)\mapsto \int_\T u(z)\overline{h(z)} \darc{z}$ to obtain a bounded bilinear form on $\AD$\/, which we denote by~$D$\/. It can then be easily verified that $D\in\Der(\AD)$\/, and that $\sE(D)=\overline{h}$ as required.

It remains only to prove the inequality \eqref{eq:keybound}. Thus, let $f$ and $g$ be holomorphic on a neighbourhood of $\Delta$\/. It is immediate from the definition of $\norm{\cdot}_{(3)}$ that $\norm{u}_{(3)}\leq\norm{f}_{(3)}\snorm{g}$\/, and since the norms $\norm{\cdot}_{(3)}$ and $\norm{\cdot}_{(2)}$ are equivalent, there exists a constant $K_1$ -- independent of $f$\/, $g$ and $h$ -- such that
\[ \norm{u}_{(2)} \leq K_1 \norm{f}_{(2)}\snorm{g} \,.\]
The desired inequality~\eqref{eq:keybound} now follows, since $\norm{f}_{(2)} \leq \snorm{f}$\/.
\end{proof}

\begin{rem}
It is possible to prove \eqref{eq:keybound} more directly, using arguments from the \emph{proof} that the seminorms $\norm{\cdot}_{(2)}$ and $\norm{\cdot}_{(3)}$ are equivalent; see Remark~\ref{rem:could-have} below.
\end{rem}

\end{section}

\begin{section}{Compactness}
It is shown in~\cite{SEM_PhD} that every bounded derivation from the disc algebra to its dual is weakly compact. (In fact, by results of Bourgain, every bounded linear map from $\AD$ to its dual is $2$-summing, and hence weakly compact, but the 
proof for general linear maps is significantly harder than the special case of derivations. For relevant details, see~\cite[\S III.I]{Woj} for instance.)

We show that rather more is true: each derivation can be 
approximated in the operator norm by finite rank derivations and so, in particular, is compact. This answers a question raised in~\cite{SEM_PhD}.
The key observation is the following.

\begin{lem}\label{l:polynomial-symbol}
Let $h$ be an analytic polynomial with zero constant term. Then $\sE^{-1}(\overline h)$ has finite rank.
\end{lem}
\begin{proof}
It suffices by linearity to consider the case where $h(z)=z^n$ for some $n\geq 1$\/. Let $D\defeq \sE^{-1}(\overline{h})$\/; then
\[ D(f)(g) = \int_\T f'(z) g(z) z^{-n} \darc{z} \,, \]
showing that $D(f)=0$ for every $f$ in the finite-codimension subspace $z^{n+2}\AD$\/. 
\end{proof}

Since the polynomials are dense in $H^1$\/, Theorem \ref{t:characterise} and Lemma~\ref{l:polynomial-symbol} give the following corollary.

\begin{cor}
Each $D\in\Der(\AD)$ is approximable and hence compact.
\end{cor}

\begin{rem}
Given $D\in\Der(\AD)$, we can construct an \emph{explicit} approximating sequence of finite-rank derivations -- just use one's favourite method of approximating $H^1$ functions by analytic polynomials, such as convolution with the F\'ejer kernel for example. This may be useful for studying other operator-theoretic properties of~$D$\/; for instance, its singular values can be estimated by using known approximation results for $H^1$-functions.
\end{rem}

\end{section}



\begin{section}{$2$-summing norms}
As remarked in the previous section, we know that every $D\in\Der(\AD)$ is $2$-summing. Hence, by Pietsch's factorization theorem, there exists a finite, positive Borel measure $\mu_D$ on $\Delta$ such that
\begin{equation}\label{eq:Pietsch}
\tag{$\star$} \norm{D(f)} \leq \norm{f}_{L^2(\mu)} = \left(\int_{\Delta} \abs{f(z)}^2 \d{\mu}_D(z)\right)^{1/2} \qquad\text{ for all $f\in \AD$\/.}
\end{equation}
However, these abstract arguments give us no way to find an `explicit' measure $\mu_D$\/. In this section we shall show how this can be done, using ideas extracted from existing proofs that the seminorms $\norm{\cdot}_{(2)}$ and $\norm{\cdot}_{(3)}$ are equivalent. See \cite[Ch. 3 and 4]{Gar_BAF}.

As in the proof of Theorem~\ref{t:characterise}, for each $h\in H^1_0$ we define a bilinear functional $D_h$ on the space of polynomials as follows: given polynomials $f$ and $g$\/, let $u$ be the unique polynomial satisfying $u(0)=0$ and $f'g=u'$\/; then set
\[
D_h(f)(g)\defeq \int_\T u(z) \overline{h(z)} \darc{z}\,.
\]

By Proposition~\ref{t:butsurj}, every $D\in\Der(\AD)$ has the form $D_h$ for some $h\in H^1_0$\/.
We may write $h(z) = \alpha z + z^2 F(z)$ where $F\in H^1$ and $\alpha\in\Cplx$.
Then, since $F+\norm{F}_1$ and $F-\norm{F}_1$ are $H^1$-functions vanishing nowhere on $\Disc$, they have square roots in $H^2$, and we can therefore write
\[ z^2F(z) =
 \left[z \left( \frac{F+\norm{F}_1}{2}\right)^{1/2}\right]^2 +
 \left[z \left( \frac{F-\norm{F}_1}{2}\right)^{1/2}\right]^2 = k_1^2 +k_2^2\,,\text{ say,} \]
where $k_1,k_2\in H^2_0$.
Thus
\begin{equation}\label{eq:divide-and-conquer}
D_h=\alpha D_Z+D_{k_1^2} +D_{k_2^2}\,,
\end{equation}
and we shall estimate these three terms separately.

\subsection*{Estimating $\norm{D_Z}$.}
Since
\[ 
D_Z(f)(g) = \int_0^{2\pi} u(e^{i\theta}) e^{-i\theta}\d{\theta}
 = 2\pi f'(0)g(0) = 2\pi\widehat{f}(1)g(0), \]
we find that
\begin{equation}\label{eq:facile}
\tag{$\diamondsuit$}
\norm{D_Z(f)} = 2\pi \abs{\widehat{f}(1)} \leq \left(2\pi \int_\T \abs{f(z)}^2\darc{z}\right)^{1/2} = \sqrt{2\pi}\norm{f}_{L^2(\lambda)}\,,
\end{equation} 
where $\lambda$ denotes arc length measure on $\T$, regarded as a Borel measure on $\Delta$\/.

\subsection*{Estimating $\norm{D_{k^2}(f)}$}
Fix $k\in H^2_0$\/. We shall construct finite, positive, Borel measures $\mu_1$ and $\mu_2$ on $\Delta$\/, such that
\begin{equation}\label{eq:le-but}
\tag{$\spadesuit$}
\norm{D_{k^2}(f)} \leq \norm{f}_{L^2(\mu_1)} + \norm{f}_{L^2(\mu_2)}
 \qquad\text{ for all polynomials $f$\/.}
\end{equation}

The key is a well-known identity of Littlewood-Paley type. To state it, and for later convenience, we introduce some notation. Let $\Lambda$ denote the measure on $\Disc\setminus\{0\}$ defined by $\d{\Lambda}(z) = 4\logweight{z}$\/.
We denote the corresponding inner product on $C(\Disc)$ by $\pair{\cdot}{\cdot}_{\Lambda}$ and the induced Hilbert-space norm by $\norm{\cdot}_{\Lambda}$\/.

\begin{lem}\label{l:LPtrick}
Let $u,v\in H^2$\/.
Then
\[
\pair{u'}{v'}_\Lambda
 = \int_{\T} (u(\zeta)-u(0))\overline{(v(\zeta)-v(0))}\darc{\zeta}
\]
In particular,
\begin{equation}\label{eq:norms-equal}
\norm{u'}_\Lambda
 = \int_{\T} \abs{(u(\zeta)-u(0)}^2\darc{\zeta} \leq \norm{u}_2\/.
\end{equation}
\end{lem}

The lemma can be proved either by using Fourier series and Parseval's identity, or using Green's formula; see \cite[Lemma 3.1]{Gar_BAF} for a proof using the latter approach.

\begin{proof}[Proof of the inequality \eqref{eq:le-but}]
Applying Lemma~\ref{l:LPtrick} yields
\begin{subequations}
\begin{equation}\label{eq:first}
D_{k^2}(f)(g)  = \int_{\T} u(\zeta)\overline{h(\zeta)} \darc{\zeta} =\pair{u'}{(k^2)'}_\Lambda = \pair{f'g}{2k'k}_\Lambda \,.
\end{equation}
Then
\begin{equation}\label{eq:second}
 \begin{aligned}
  \Abs{\pair{f'g}{2kk'}_\Lambda } 
 & \leq 2\int_\D \abs{(f'gk'k)(z)} \d{\Lambda}(z)  \\
 & \leq 2 \snorm{g} \pair{\abs{f'k}}{\abs{k'}}_\Lambda  \\
 & \leq 2 \snorm{g} \norm{k'}_\Lambda \norm{f'k}_\Lambda & \quad\text{(Cauchy-Schwarz)} \\
 & = 2\snorm{g}\norm{k}_2 \norm{f'k}_\Lambda &\quad\text{(Equation~\eqref{eq:norms-equal})}
\end{aligned}
\end{equation}

We now consider $\norm{f'k}_\Lambda$\/. 
Since $f'k = (fk)'-fk'$\/, we have
\begin{equation}\label{eq:third}
\norm{f'k}_\Lambda \leq \norm{(fk)'}_\Lambda + \norm{fk'}_\Lambda
  = \norm{fk}_2 + \norm{fk'}_\Lambda 
  = \norm{f}_{L^2(\mu_B)} + \norm{f}_{L^2(\mu_I)}
\end{equation}
\end{subequations}
where $\mu_B$ and $\mu_I$ are the positive Borel measures on $\Disc$ given by
\[ \mu_B(X) \defeq \int_{X\cap\T} \abs{k(z)}^2 \darc{z} \quad;\quad \mu_I(X) \defeq \int_{X\cap\Disc} \abs{k'(z)}^2 \d{\Lambda}(z)\,. \]
Clearly $\mu_B$ is finite, since $k\in H^2$\/, and $\mu_I$ is finite since
\[  \mu_I(\Delta) = \int_\D \abs{k'(z)}^2 \d{\Lambda}(z) = \norm{k'}_\Lambda^2 = \norm{k}_2^2 <\infty\,. \]

Putting the inequalities \eqref{eq:first}, \eqref{eq:second} and \eqref{eq:third} together, we have
\begin{equation}\label{eq:direct}
\abs{D_{k^2}(f)(g)} \leq 4\norm{g}_\infty \norm{k}_2 \left( \norm{f}_{L^2(\mu_B)} + \norm{f}_{L^2(\mu_I)} \right)\,,
\end{equation}
which on rescaling gives us the desired inequality.
\end{proof}

Finally, using \eqref{eq:divide-and-conquer}, and combining the estimates \eqref{eq:facile} and \eqref{eq:le-but}, we obtain finite, positive Borel measures $\mu_1,\dots,\mu_5$ such that
\[ \norm{D_h(f)} \leq \sum_{j=1}^5 \norm{f}_{L^2(\mu_j)} \qquad\text{ for all polynomials $f$\/.} \]
To get a single $L^2$-upper bound, observe that there is a crude estimate
\begin{equation}
\label{eq:combining}
\sum_{j=1}^5 \norm{f}_{L^2(\mu_j)} \leq \sqrt{5} \norm{f}_{L^2(\mu_1+\dots+\mu_5)}\,,
\end{equation}
which is a straightforward consequence of the elementary inequality
\[ \left(\sum_{j=1}^m a_j\right)^2 = \sum_{j=1}^m\sum_{k=1}^m a_ja_k \leq \frac{1}{2} \sum_{j=1}^m\sum_{k=1}^m (a_j^2+a_k^2) \leq m\sum_{j=1}^m a_j^2 \qquad\text{for all $a_1,\ldots,a_m\geq 0$\/.} \]
Thus in \eqref{eq:Pietsch}, we may take $\mu_D=\sqrt{5}(\mu_1+\dots+\mu_5)$\/. Although we have not written out an explicit formula for $\mu_D$\/, it is clear from the calculations above that it could be easily done for a given $h\in H^1_0$\/.

\begin{rem}\label{rem:could-have}
The inequality \eqref{eq:direct} implies, in particular, that $D_{k^2}\in\Der(\AD)$ with $\norm{D_{k^2}} \leq 8\norm{k}_2^2$\/. One can use this to obtain a direct proof of the earlier inequality~\eqref{eq:keybound}, which does not rely on first knowing the equivalence of the seminorms $\norm{\cdot}_{(2)}$ and $\norm{\cdot}_{(3)}$. Thus Theorem~\ref{t:characterise} could have been proved without explicit mention of BMOA, although the argument would then have been longer and more opaque.
\end{rem}

\end{section}
\providecommand{\bysame}{\leavevmode\hbox to3em{\hrulefill}\thinspace}
\providecommand{\MR}{\relax\ifhmode\unskip\space\fi MR }
\renewcommand{\MRhref}[2]{%
  \href{http://www.ams.org/mathscinet-getitem?mr=#1}{#2}
}

\newcommand{\newMR}[2]{\hfill\MRhref{#1}{[MR #2]}}

\end{document}